\documentclass[12pt]{amsart}

\usepackage{amsmath}
\usepackage{amssymb}
\usepackage{latexsym}
\usepackage{enumerate}
\usepackage{amscd}
\usepackage{sseq}

\usepackage[all]{xy}
\usepackage{graphics}
\usepackage[active]{srcltx}

\newcommand{\HH}{\mathrm{H}}
\newcommand{\KR}{\mathrm{KR}}
\newcommand{\KU}{\mathrm{KU}}
\newcommand{\KO}{\mathrm{KO}}
\newcommand{\Hom}{\mathrm{Hom}}

\newcommand{\Z}{\mathbf{Z}}

\newcommand{\R}{\mathbf{R}}
\newcommand{\C}{\mathbf{C}}
\newcommand{\F}{\mathbf{F}}

\newtheorem{theorem}{Theorem}[section]
\newtheorem{proposition}[theorem]{Proposition}
\newtheorem{corollary}[theorem]{Corollary}
\newtheorem{lemma}[theorem]{Lemma}
\newtheorem{definition}[theorem]{Definition}
\newtheorem{remark}[theorem]{Remark}
\newtheorem{example}[theorem]{Example}

\newcommand{\wmono}{ \ar@{>->}[r]}
\newcommand{\wmonovert}{ \ar@{>->}[d]}
\newcommand{\cof}{ \ar@{^{(}->}[r]}

\begin{document}

\title{Conjugation spaces and equivariant Chern classes}

\author{Wolfgang Pitsch}
\address{Universitat Aut\`onoma de Barcelona \\ Departament de Matem\`atiques\\
E-08193 Bellaterra, Spain}
\email{pitsch@mat.uab.es}

\author{J\'er\^ome Scherer}
\address{EPFL \\ SB MATHGEOM\\ Station 8, MA B3 455 
CH-1015 Lausanne, Switzerland}
\email{jerome.scherer@epfl.ch}

\thanks{The authors are partially supported  by FEDER/MEC grant MTM2010-20692. The second author would
like to thank the MPI in Bonn for its hospitality.}



\begin{abstract}
Let $\eta$ be a Real bundle, in the sense of Atiyah, over a space $X$.
This is a complex vector bundle together with an involution which is 
compatible with complex conjugation. We use the fact that $BU$ has a canonical  
structure of a conjugation space, as defined by Hausmann, 
Holm, and Puppe, to construct equivariant Chern classes in certain 
equivariant cohomology groups of $X$ with twisted integer coefficients. 
We show that these classes determine the (non-equivariant) Chern classes of 
$\eta$, forgetting the involution on $X$, and the Stiefel-Whitney classes of the 
real bundle of fixed points.
\end{abstract}

\maketitle

\section*{Introduction}
Let $X$ be a topological space and $\eta$ a complex bundle over $X$. A \emph{Real structure} on $\eta$, in the sense of Atiyah \cite{MR0206940}, 
is given by a couple of compatible involutions on $X$ and on the total space of $\eta$ such that the later is complex anti-linear. 
The machinery of K-theory applies to these bundles and yields $\KR(X)$, called Real K-theory. It comes equipped 
with two natural transformations: the forgetful map $\KR (X) \rightarrow \KU (X)$ and the ``fixed points map'' $\KR (X) \rightarrow \KO (X^\tau)$, 
where $X^\tau$ denotes the subspace of fixed points. Since its introduction by Atiyah, Real K-theory has proved to be a useful tool 
that interpolates between complex K-theory and ordinary real K-theory. For instance it allows one to prove in a unified way both 
Bott periodicity phenomena. It is therefore natural to search for a theory of characteristic classes for Real bundles, called
equivariant Chern classes for reasons that will be clear from their construction, and, ideally, these should interpolate between ordinary 
Chern classes and Stiefel-Whitney classes. 

Kahn provided such a construction in \cite{MR877194}. The equivariant Chern classes $\widetilde{c}_n(\eta)$
live in the equivariant cohomology $\HH^{2n}_{C_2}(X;\Z(n))$ with twisted integral coefficients,
where the action of the cyclic group  of order $2$ on $\Z(n)$ is by
multiplication by $(-1)^n$. When the involution on $X$ is trivial, he shows that the mod $2$ reduction of the equivariant
Chern classes are the Stiefel-Whitney classes of $\eta$. His ideas have found applications in algebro-geometric contexts, such
as the work of dos Santos and Lima-Filho, \cite{MR2084395}.

In this article we adopt a slightly different point of view. We take advantage of the existence of universal Real bundles, namely 
the canonical bundles over $BU(n)$ with involution induced by complex conjugation, to carry out our construction. The key
ingredient which makes the analysis of the situation quite elementary is that $BU(n)$ is a so-called \emph{spherical conjugation space},
a notion introduced by Hausmann, Holm, and Puppe in \cite{MR2171799}. It turns out that the conjugation structure, which
lives at the level of mod $2$ cohomology, upgrades to the level of the (correct choice of) twisted integral coefficients.
The bundle of fixed points in the universal Real bundle is the universal bundle for real K-theory with base space $BO(n)$, 
and this explains the relationship between our new characteristic classes and the Stiefel-Whitney classes.

\medskip

\noindent
{\bf Proposition~\ref{propChernSW2}.}
{\em Let $\eta$ be a Real bundle over $X$ and $\eta^\tau$ the associated real bundle of fixed points. Then the image of 
$\widetilde{c}_n(\eta)$ in $\HH^\ast (X^\tau; \mathbf{F}_2)[u]$ is $\displaystyle\sum_{i=0}^n Sq^i(w_n(\eta^\tau))u^{n-i}$.}

\medskip

As a particular case, we get back Kahn's main result in \cite{MR877194} about Real bundles over spaces with trivial involution. 
Notwithstanding Kahn's construction, we believe that the introduction of the ``conjugation space'' structure sheds 
some more light on the properties of the equivariant Chern classes and in particular on the conceptual reasons of the
appearance of the twisting of the coefficients. 

\medskip

\noindent
{\bf Acknowledgements.} We would like to thank Jean-Claude Hausmann for introducing us to the beautiful theory of 
conjugation spaces.

\section{Notation and basic facts about conjugation spaces}
\label{sec:notation}
The cyclic group of order two is  $C_2$ and $\mathbf{F}_2$ is the field of $2$ elements. The graded ring 
$\HH^\ast(BC_2; \mathbf{Z})$ is isomorphic to the ring $\mathbf{Z}[a]/(2a)$ 
where $a$ has degree $2$ and the mod $2$ cohomology $\HH^\ast(BC_2; \F_2)$ is isomorphic to
$\F_2[u]$ a polynomial algebra on a generator $u$ of degree~$1$. If $X$ is a topological space equipped 
with a (left) $C_2$-action (often denoted by $\tau$), the Borel construction, or the homotopy orbits, $X_{hC_2}$ is the space
$EC_2 \times_{C_2} X$. The equivariant cohomology functor is defined on $C_2$-spaces as the integral
cohomology of the Borel construction (and likewise with mod $2$ coefficients): 
$\HH^\ast_{C_2}(X; \Z) \cong \HH^\ast (X_{hC_2}; \Z)$. The inclusion of the fixed points
$X^\tau = X^{C_2}$ in $X$ induces a restriction homomorphism in cohomology
$r : \HH^{\ast}_{C_2}(X; \F_2) \rightarrow \HH^{\ast}_{C_2}(X^\tau; \F_2)$. The latter is
isomorphic to $\HH^\ast(X; \F_2) \otimes \F_2[u]$ since the action of $C_2$ is
trivial on the fixed points.

Let us give now the basic facts about conjugation spaces which we will need to construct equivariant
Chern classes. A \emph{conjugation space} is a $C_2$-space together with an $\HH^*$-frame $(\kappa,\sigma)$, i.e.
\begin{enumerate}
 \item[a)] an additive isomorphism $\kappa : \HH^{2\ast}(X;\F_2) \rightarrow \HH^\ast(X^\tau; \F_2)$
dividing degrees by~$2$,

\item[b)] an additive section $\sigma : \HH^{2\ast}(X;\F_2) \rightarrow \HH^{2\ast}(X_{hC_2}; \F_2)$
of the restriction map $\rho : \HH^{2\ast}(X_{hC_2}; \F_2) \rightarrow \HH^{2\ast}(X; \F_2)$,
\end{enumerate}

which satisfy the conjugation equation:
\[
r \circ \sigma(x) = \kappa(x) u^m + lt_m
\]
for all $x \in H^{2m}(X; \F_2)$ and all $m \in \mathbf{N}$, where $lt_m$ is a polynomial in the variable $u$ of 
degree less than $m$.

A \emph{conjugation cell} is a $C_2$-space which is equivariantly homeomorphic to
the unit disk in $\mathbf{C}^n$ equipped with complex conjugation.
A \emph{spherical conjugation complex} is a CW-complex $X$ constructed from conjugation cells with 
equivariant attaching maps. Hausmann, Holm, and Puppe prove in their foundational article \cite{MR2171799}
that any spherical conjugation complex is a conjugation space.
In fact, the original motivating examples of conjugation spaces are the complex Grassmanians 
$Gr_{n,k}(\mathbf{C})$, which are spherical. In particular the spaces $BU(n)$, for any $n\geq 1$, and $BU$ are
spherical conjugation spaces. The main property of an $\HH^*$-frame that we will keep using in this article is
the following.

\begin{theorem}[Hausmann, Holm, Puppe \cite{MR2171799}]
The morphisms $\kappa$ and $\sigma$ in an $\HH^*$-frame are ring homomorphisms.
\end{theorem}

\section{Why one has to twist the coefficients}
We wish to construct equivariant Chern classes for Real bundles in equivariant cohomology groups with 
integral coefficients and our main requirement is that one recovers the classical Chern classes by forgetting the $C_2$-action. 
Thus, a first na\"ive, but natural, place to look for such classes is in the integral equivariant cohomology 
 $\HH^\ast_{C_2}(X;\Z)$ with \emph{trivial} coefficients. We will illustrate in a fundamental example
 why this does not work. It will show us at the same time how to calibrate the correct answer.

\begin{example}
\label{Hopf}
{\rm
Consider the sphere $S^2$ equipped with the reflection through the equatorial circle. This is a spherical conjugation space obtained
by identifying the boundary of a conjugation disc to a point. The Hopf bundle over $S^2$ is naturally a Real bundle, so its first equivariant
Chern class should correspond to a copy of the integers in $\HH^2_{C_2}(S^2;\Z) = \HH^2\left( (S^2)_{hC_2}; \Z \right)$.

But in this case it is easy to identify the Borel construction, since $S^2$ consists of an equatorial copy of $S^1$ with trivial action and two discs, the
hemispheres, transposed by $\tau$. As a consequence $S^2_{hC_2}$ is equivalent to the half-smash $S^1 \ltimes \R P^\infty$, hence
$\HH^2_{C_2}(S^2;\Z) \cong \Z/2$. As there is no non-trivial homomorphism from $\Z/2$ to $\HH^2(S^2; \Z) \cong \Z$, there is no
way to recover $c_2(\eta)$ from any class in $\HH^2_{C_2}(S^2;\Z)$.

However, the Hopf bundle over $S^4$ can also be seen as a Real bundle over the conjugation sphere of dimension~$4$
and one computes here that $\HH^4_{C_2}(S^4; \Z) \cong \Z \oplus \Z/2$. There is room here for an equivariant Chern class!
}
\end{example}

A closer inspection of the examples shows that the difference can be stated as follows: on cells of dimension 
$2 \text{ mod } 4$ the conjugation reverses the orientation and on cells of dimension $ 0 \text{ mod } 4$ it preserves 
the orientation\footnote{A conjugation cell of dimension $2n$  can be viewed as an open unit disk in $\C^n$ with a 
$C_2$-action induced by complex conjugation, in particular it has a canonical orientation.}. This leads us to look for 
even degree Chern classes in the ordinary equivariant cohomology group $\HH^\ast_{C_2}(X; \Z)$ and for odd degree   
Chern classes in the equivariant cohomology $\HH^\ast_{C_2}(X;\Z(1))$ with twisted coefficients $\Z(1)$, the  
$C_2$-module  $\Z$ endowed with the  change of sign action.

Let us denote by $\Z(n)$ the $C_2$-module $\Z$ where the action is given by multiplication by $(-1)^n$. It is straightforward to see 
that $\Hom_\Z(\Z(i),\Z(j)) \cong \Z(i+j)$ and $\Z(i)\otimes_\Z \Z(j) \cong \Z(i+j)$. Because the module structure depends only on the parity of $n$,
we only keep the modules $\Z(0)$ and $\Z(1)$ and notice the following.

\begin{remark}
{\rm 
Let $X$ be a spherical conjugation complex. The cup product turns the direct sum 
$\HH^{\ast}_{C_2}(X;\Z(0)) \oplus \HH^{\ast}_{C_2}(X;\Z(1))$ into a commutative algebra, which is
natural in $X$ with respect to equivariant maps. We observe that the submodule
\[
\mathcal{H}^\ast_{C_2}(X;{}^t\Z) =  \bigoplus_{n \geq 0}\HH^{4n}_{C_2}(X;\Z(0))\bigoplus \bigoplus_{n \geq 0}\HH^{4n+2}_{C_2}(X;\Z(1))
\]
is a subalgebra.}
\end{remark}

We turn now to a more detailed analysis of the algebra  $\mathcal{H}^\ast_{C_2}(X,{}^t\Z)$ for a spherical 
conjugation complex.  For this we compute the Leray-Serre spectral sequences with twisted coefficients $\Z(0)$ and $\Z(1)$ associated to the canonical fibration 
\[
\xymatrix{X \cof &  X_{hC_2} \ar@{->>}[r] & BC_2}.
\]
We have to compute first the cohomology $\HH^\ast(X, \Z(n))$ as $C_2$-modules, \cite[Section VII.7]{MR1324339}, where the action
of $C_2$ is the diagonal one, induced by the action on $C^{cell}_*(X)$ and $\Z(n)$. We write $\HH^n(X;\Z)$ for the non-equivariant
cohomology of $X$ with trivial action. A spherical conjugation complex has only even dimensional cells by construction so that one finds:
\[
\HH^n(X;\Z(0)) = \left\{
\begin{array}{cl}
 \Z(0) & \text{ if } n = 0, \\
 0 & \text{ if } n \text{ is odd,} \\
\HH^n(X;\Z) \otimes \Z(0) & \text{ if } n = 0 \text{ mod } 4,\\
\HH^n(X;\Z) \otimes \Z(1) & \text{ if } n = 2 \text{ mod } 4.
\end{array}
\right.
\]
\[
\HH^n(X;\Z(1)) = \left\{
\begin{array}{cl}
 \Z(1) & \text{ if } n = 0,\\
 0 & \text{ if } n \text{ is odd,} \\
\HH^n(X;\Z) \otimes \Z(1) & \text{ if } n = 0 \text{ mod } 4,\\
\HH^n(X;\Z) \otimes \Z(0) & \text{ if } n = 2 \text{ mod } 4.
\end{array}
\right.
\]
As for the cohomology of $BC_2$ one has $\HH^\ast(BC_2;\Z(0)) \cong \Z[a]/(2a)$ where $a$ has degree~$2$, and 
$\HH^\ast(BC_2;\Z(1)) \cong \oplus_{i \geq 0}\Z/2 \langle e_{2i+1} \rangle$ where $e_{2i+1}$ has degree $2i+1$. Both 
spectral sequences have a vertical periodicity pattern of order $4$ starting from the first line on.

\begin{table}[htb]
\[
\begin{array}{r|c|c|c|c|c || l}
q \\
0 \text{ mod } 4&\oplus \Z & 0 & \oplus \Z/2 & 0 & \oplus \Z/2 & \HH^4(X;\Z) \otimes \HH^p(BC_2; \Z) \\
\text{odd}  & 0 & 0 & 0 & 0 & 0 & \\
2 \text{ mod } 4& 0 & \oplus \Z/2 & 0 & \oplus \Z/2 & 0 &\\
\text{odd}& 0 & 0 & 0 & 0 & 0 & \\
\hline\hline
0 & \Z & 0 & \Z/2 & 0 & \Z/2 & \HH^0(X;\Z)\otimes \HH^p(BC_2;\Z) \\ \hline
 & 0& 1& 2 & 3 & 4 & p
\end{array}
\]
\caption{ Spectral sequence for coefficients  $\Z(0)$}
\end{table}

\begin{table}[htb]
\[
\begin{array}{r|c|c|c|c|c || l}
q \\
0 \text{ mod } 4 & 0 & \oplus \Z/2 & 0 & \oplus \Z/2 & 0 &\\
\text{ odd } & 0 & 0 & 0 & 0 & 0 & \\
2 \text{ mod } 4&\oplus \Z & 0 & \oplus \Z/2 & 0 & \oplus \Z/2 & \HH^2(X;\Z) \otimes \HH^p(BC_2; \Z) \\
\text{ odd }& 0 & 0 & 0 & 0 & 0 & \\
\hline \hline
0 & 0 & \Z/2 & 0 & \Z/2 & 0 & \\ \hline
 & 0& 1& 2 & 3 & 4 &
\end{array}
\]
\caption{ Spectral sequence for coefficients  $\Z(1)$}
\end{table}

\begin{lemma}
\label{lem:SS}
The two spectral sequences $\HH^p(BC_2;\HH^q(X;\Z(1))) \Rightarrow \HH^{p+q}_{C_2}(X; \Z(1))$ and 
$\HH^p(BC_2;\HH^q(X;\Z(0))) \Rightarrow \HH^{p+q}_{C_2}(X; \Z(0))$ degenerate at the page $E_2$. Moreover,
$\HH^{2n}_{C_2}(X; \Z(n)) \cong \HH^{2n}(X; \Z) \oplus \displaystyle\bigoplus_{p<n} \HH^{2p}(X; \Z) \otimes \Z/2$.
\end{lemma}
\begin{proof}
Reducing the coefficients modulo $2$ gives a natural transformation of spectral sequences from either of the spectral sequences 
to the spectral sequence with trivial coefficients $\F_2$, converging to $\HH^\ast(X_{hC_2}; \F_2)$. As $X$ is a conjugation space 
we know that the latter degenerates at the page $E_2$.

A direct computation, using the description of $\HH^\ast(X;\Z(n))$, shows that on non-zero elements, mod $2$ 
reduction induces an isomorphism for any $p >1$: 
\[
 \HH^p(BC_2;\HH^q(X;\Z(i))) \stackrel{\cong}{\rightarrow}  \HH^p(BC_2;\HH^q(X;\F_2)).
\]
Since on the right hand side we have permanent cycles the left hand side has to be made of permanent cycles too. In
particular all differentials originated in the vertical axis, $p=0$, must be trivial, so that $ \HH^0(BC_2;\HH^q(X;\Z(n)))$
is also made of permanent cycles.

The comparison with the mod $2$ reduction also allows us to compute the module structure of $\HH^{2n}_{C_2}(X, \Z(n))$.
The surjection of the associated graded modules
\[
\HH^{2n}(X; \Z) \oplus \displaystyle\bigoplus_{p<n} \HH^{2p}(X; \Z) \otimes \Z/2 \twoheadrightarrow \HH^{2n}(X; \F_2) \oplus \displaystyle\bigoplus_{p<n} \HH^{2p}(X; \F_2)
\] 
is an isomorphism except on the first factor. The latter graded module is isomorphic to $\HH^{2n}_{C_2}(X; \F_2)$
as there are no non-trivial extensions of $\F_2$-vector spaces, and thus there can be no non-trivial extensions
in the former (a copy of $\Z/2^k$ for some $k > 1$ would imply the presence of $\Z/2^{k-1}$ in the kernel of the mod $2$ 
reduction map, which is impossible). 
\end{proof}

Combining the edge homomorphisms of both spectral sequences we get a restriction homomorphism to ordinary cohomology.
\begin{lemma}
The restriction map $\widetilde{\rho} : \mathcal{H}^\ast_{C_2}(X,{}^t\Z) \rightarrow \HH^\ast (X;\Z)$ is  surjective and 
a ring homomorphism. \hfill{\qed}
\end{lemma}

For a spherical conjugation complex there is a canonical integral lift for the cohomological section encompassed in the $\HH^*$-frame.

\begin{theorem}
\label{thm intliftsigma}
Let $X$ be a spherical conjugation complex with $\HH^*$-frame $(\sigma,\kappa)$. Then there is a unique section 
$\widetilde{\sigma}$ to the restriction map $\widetilde{\rho} : \mathcal{H}^\ast_{C_2}(X,{}^t\Z) \rightarrow \HH^\ast(X;\Z)$ such that the 
mod $2$ reduction of the section is $\sigma$. Moreover the section  $\widetilde{\sigma}$ is a ring homomorphism which is 
natural with respect to equivariant maps between spherical complexes.
\end{theorem}

\begin{proof}
As $X$ has only cells in even dimension, $\HH^\ast(X;\Z)$ is free in each degree. We have an isomorphism of modules
by Lemma~\ref{lem:SS}
\[
\mathcal{H}^\ast_{C_2}(X,{}^t\Z) \cong \HH^*(BC_2;\Z) \otimes \HH^*(X;\Z) \cong \Z[a]/(2a) \otimes \HH^*(X;\Z)\, .
\]
We also know that $\HH^\ast_{C_2}(X; \F_2)\cong \HH^*(BC_2; \F_2)  \otimes \HH^\ast(X; \F_2) \cong \F_2[u] \otimes \HH^\ast(X; \F_2)$
since $X$ is a conjugation space.

Let $K^\ast$ be the kernel of the (surjective) map $\HH^\ast(X;\Z) \rightarrow \HH^\ast(X;\F_2)$.
Since the mod $2$ reduction induces isomorphisms $\HH^\ast(BC_2; \Z) \to \F_2[u]$ in any positive even degree,
we have a commutative diagram of groups:
\[
\xymatrix{
0 \ar[r] & K^\ast  \ar[r] \ar@{=}[d]& \mathcal{H}^\ast_{C_2}(X,{}^t\Z) \ar^{\text{red}}[r] \ar@{->>}^{\widetilde{\rho}}[d] 
& \HH^{2\ast}_{C_2}(X;\mathbf{F}_2) \ar[r]  \ar@{->>}^{\rho}[d] & 0  \\
0 \ar[r] & K^\ast \ar[r] & \HH^{\ast}(X;\mathbf{Z}) \ar[r] & \HH^{\ast}(X;\mathbf{F}_2) \ar[r] & 0,
}
\]
which implies that the right-hand square is a pull-back square. The section $\sigma$ splitting the
restriction $\rho$ determines hence a unique section $\widetilde{\sigma}: \HH^{\ast}(X;\mathbf{Z})
\to \mathcal{H}^\ast_{C_2}(X,{}^t\Z)$. It is a ring homomorphism because both $\sigma$ and the mod $2$
reduction are so. Finally, naturality follows from that of the section $\sigma$ in an $\HH^\ast$-frame and the fact
that the pullback diagram is natural in $X$.
\end{proof}

A little more can be extracted from the proof: the collapse of the two spectral sequences allows one to prove 
a version of the classical Leray-Hirsch theorem for the algebra $\mathcal{H}^\ast_{C_2}(X,{}^t\Z)$. The fibration
$X_{hC_2} \rightarrow BC_2$ splits since the fixed point set $X^\tau$ is not empty. As it is in fact connected,
any two sections are homotopic, which determines a canonical copy of $\HH^\ast(BC_2;\Z)$ in
$\mathcal{H}^\ast_{C_2}(X,{}^t\Z)$.

\begin{corollary}[Leray-Hirsh for the equivariant cohomology]
\label{cor lerayhirsh}
Let $X$ be a spherical conjugation space, then  the canonical map
\[
\begin{array}{rcl}
\Z[a]/(2a) \otimes\HH^\ast(X;\Z) & \longrightarrow & \mathcal{H}^\ast_{C_2}(X,{}^t\Z) \\
r \otimes x  & \longmapsto & r \cup  \widetilde{\sigma}(x)
\end{array}
\]
is a $\Z[a]/(2a)$-algebra isomorphism, natural in $X$. \hfill{\qed}
\end{corollary}

\section{Equivariant cohomology computations for $BU$}
A Real bundle (in the sense of Atiyah) \cite{MR0206940} over a space $X$ equipped with an 
involution $\tau$, is a complex bundle $\eta$ over $X$ together with an involution on the total 
space, compatible with $\tau$ and which is anti-linear on the fibers. The tautological bundle over $BU(n)$
is the universal Real bundle of rank $n$, and the  tautological bundle over $BU$ is a universal stable 
bundle, where the conjugation on the base space is induced by complex conjugation on the coefficients 
of complex matrices (see for instance \cite{MR639802} for an explicit proof of this fact). 
Our definition of equivariant Chern classes will  rest on the definition of the universal equivariant Chern classes as 
elements in $\mathcal{H}^\ast_{C_2}(BU,{}^t\Z)$. We will thus need a good
understanding of how this algebra behaves with respect to the Whitney sum and
restriction to fixed points.
Theorem~\ref{thm intliftsigma} applies in particular to the spherical conjugation complex $BU$, equipped
with the $C_2$-action coming from complex conjugation. Recall that the ordinary cohomology
$\HH^\ast(BU;\Z)$ is a polynomial algebra in the ordinary universal Chern classes $c_n$. 

\begin{definition}
\label{def:equivChern}
{\rm The classes $\widetilde c_n = \widetilde \sigma(c_n) \in \HH^{2n}_{C_2}(BU; \Z(n))$ are
the \emph{universal equivariant Chern classes}.
}
\end{definition}

{}From Corollary \ref{cor lerayhirsh} we get the expected structure for the equivariant cohomology algebra
$\mathcal{H}^\ast_{C_2}(BU,{}^t\Z)$: It is isomorphic to a polynomial algebra over $\Z[a]/(2a)$ on  the equivariant Chern 
classes $\widetilde c_n$ for $n \geq 1$. This result is the analogue of Kahn's \cite[Th\'eor\`eme~3]{MR877194}.

\begin{proposition}
\label{prop:equivBU}
The equivariant cohomology algebra $\mathcal{H}^\ast_{C_2}(BU,{}^t\Z)$ is isomorphic to $\Z[a, \widetilde c_1, \widetilde c_2, \dots ]/(2a)$. 
\hfill{\qed}
\end{proposition}

Our next result concerns the relationship between the equivariant Chern classes living in
the equivariant cohomology of $BU$ and the Stiefel-Whitney classes in the mod $2$ cohomology
of $BO$. The subspace of fixed points in $BU$ under complex conjugation is $BU^\tau = BO$.
Reducing the coefficients in the last cohomology algebra modulo $2$
yields a map of equivariant homology groups $ \mathcal{H}^\ast_{C_2}(BU,{}^t\Z) \rightarrow \HH^\ast_{C_2}(BU; \F_2)$, 
and restricting then to the fixed points we get a homomorphism:
\[
\bar r: \mathcal{H}^\ast_{C_2}(BU,{}^t\Z) \rightarrow \HH^\ast(BO;\mathbf{F}_2)[u].
\]
The conjugation equation allows us to compute the image of the equivariant Chern classes.

\begin{proposition}
\label{propChernSW}
The image of $\widetilde{c}_n$ in $\HH^\ast(BO;\mathbf{F}_2)[u]$ is $\displaystyle \sum_{i=0}^n Sq^i(w_n)u^{n-i}$.
\end{proposition}

\begin{proof}
Let us write $\bar c_i$ for the mod $2$ reduction of the universal Chern classes. By definition of $\widetilde{\sigma}$ we have:
\[
\bar r(\widetilde{c}_n) = \bar r (\widetilde{\sigma}(c_n)) = r (\sigma(\bar c_n))
\]
Franz and Puppe, \cite[Theorem~1.1]{MR2198191},
computed for us the effect of the restriction to the fixed points on the image of the section $\sigma$. Here
\[
r (\sigma(\bar c_n)) = \sum Sq^i(\kappa (\bar c_n)u^{n-i}) = \sum Sq^i(w_n)u^{n-i}
\]
since, via the ``halving isomorphism" $\kappa$, the mod $2$ Chern classes correspond to the Stiefel-Whitney classes of the
fixed point bundle, $\kappa(\bar c_n)= w_n$, \cite[Proposition~6.8]{MR2171799}.
\end{proof}

The last result we will need is the behavior of the universal equivariant Chern classes under Whitney sum. 
Let $\mu : BU \times BU \rightarrow BU$ be the H-structure map which induces the Whitney sum on complex bundles. 
The space $BU \times BU$, under the diagonal action, is also a spherical conjugation space, 
\cite[Proposition~4.5]{MR2171799}, and $\mu$ is an equivariant map. The cross product is a morphism
\[
\times \colon \HH_{C_2}^{2n}(BU; \Z(n)) \otimes \HH_{C_2}^{2m}(BU; \Z(m)) \rightarrow \HH_{C_2 \times C_2}^{2(n+m)}(BU \times BU; \Z(n)\otimes \Z(m))
\]
and the diagonal inclusion $C_2 \rightarrow C_2 \times  C_2$ induces a map:
\[ 
 \delta \colon (BU \times BU)_{hC_2} \longrightarrow  (BU \times BU)_{h(C_2 \times C_2)} \simeq BU_{hC_2} \times BU_{hC_2}. 
\]

\begin{lemma}\label{lem factmu}
The map induced by the multiplication $\mu: BU \times BU \rightarrow BU$ sends the universal equivariant Chern 
class $\widetilde{c}_n$ to the cross product $\delta^*( \sum_{i=0}^n \widetilde{c}_i \times \widetilde{c}_{n-i})$.
\end{lemma}
\begin{proof} We have a pull-back square of $\Z[a]/(2a)$-modules, as in the proof of Theorem~\ref{thm intliftsigma}:
\[
\xymatrix{
{ \mathcal{H}^\ast_{C_2}(BU\times BU;{}^t\Z)} \ar[r]^-{\text{mod } 2} \ar[d]_{\widetilde{\rho}} & {\HH^{2\ast}_{C_2}(BU \times BU; \F_2)}
 \ar_{\rho}[d] \\ 
 {\HH^\ast(BU \times BU; \Z)} \ar[r] & {\HH^\ast(BU \times BU; \F_2)}
} 
\]
To prove the lemma it remains to show that both $\widetilde \rho$ and the mod $2$ reduction send $\mu(\widetilde{c}_n)$ 
and $\delta^*( \sum_{i=0}^n \widetilde{c}_i \times \widetilde{c}_{n-i})$ to the same element. Let us first compute the image under
$\widetilde \rho$:
\[
({\widetilde{\rho}} \circ \mu) (\widetilde{c}_n)  =  (\mu \circ {\widetilde{\rho}})(\widetilde{c}_n) = 
\mu \left( {\widetilde{\rho}}  ({\widetilde{\sigma}}(c_n)) \right)= \mu(c_n)
\]
The non-equivariant computation of $\mu(c_n)$, \cite[Theorem~1.4]{MR652459}, identifies this with the cross
product $\sum_{i=0}^n  c_i \times c_{n-i}$. Therefore, the naturality of the cross-product implies that
\[
{\widetilde{\rho}} (\mu (\widetilde{c}_n)) = \sum_{i=0}^n  c_i \times c_{n-i} = \sum_{i=0}^n \widetilde \rho(\widetilde {c}_i) \times \widetilde \rho(\widetilde{c}_{n-i})
= \widetilde \rho \left(\delta^*(\sum_{i=0}^n \widetilde{c}_i \times \widetilde{c}_{n-i})\right)\, .
\]
To compare next the images under mod $2$ reduction, we denote by $\bar c_i$ the reduction of the Chern class $c_i$.
Then, by naturality of the section $\sigma$ in an $\HH^*$-frame, we compute
\[
\delta^*( \sum_{i=0}^n \widetilde{c}_i \times \widetilde{c}_{n-i}) \text{mod } 2 = 
\delta^*( \sum_{i=0}^n \sigma(\bar{c}_i) \times \sigma(\bar{c}_{n-i})) = \sigma(\sum_{i=0}^n \bar c_i \times \bar c_{n-i})
\]
We have seen above that $\mu(c_n)$ coincides with the cross product $\sum c_i \times c_{n-i}$, and hence so
does $\mu(\bar c_n)$ with $\sum \bar c_i \times \bar c_{n-i}$. Therefore the mod $2$ reduction of the cross product
equals to
\[
\sigma(\mu(\bar c_n)) = \mu (\sigma(\bar c_n)) = \mu(\widetilde \sigma (c_n) \text{mod } 2) = \mu(\widetilde c_n) \text{mod } 2
\]
\end{proof}

\section{Classical and equivariant Chern classes}
We are now ready to introduce equivariant Chern classes for Real bundles.
Recall the axiomatic definition of Chern classes, as stated by Hirzebruch in \cite[Chap. 1, p. 66]{MR0202713}.

\begin{enumerate}
\item[I] (Existence) For every complex bundle $\eta$ over a finite dimensional paracompact space $B$ and 
every integer $i \geq 0$ there exists a \emph{Chern class} $c_i(\eta)$ in $\HH^{2i}(B;\mathbf{Z})$. 
The class $c_0(\eta)= 1$ is the unit element.

\item[II] (Naturality) If $f : B_1 \rightarrow B_2$ is a map of spaces and $\xi$ is a complex  bundle over 
$B_2$,  then $f^{\ast}(c_i(\xi)) = c_i(f^{\ast}(\xi))$ for all $i \geq 0$ .

\item[III] (Whitney sum) If $\eta = \eta_1 \boxplus \eta_2$ then $c(\eta)=c(\eta_1)c(\eta_2)$, 
where $c(-)$ is the total Chern class $\sum_{i=0}^\infty c_i(-)$.

\item[IV] (Normalization) If $\eta$ denotes the canonical bundle over $\mathbf{C}P^1$ then 
$c(\eta) = 1 + h$ where $h \in \HH^{2}(\mathbf{C}P^1;\mathbf{Z}) \cong \mathbf{Z}$ is the natural 
generator.
\end{enumerate}

We want our equivariant Chern classes to live in the equivariant cohomology 
$\mathcal{H}^\ast_{C_2}(-;{}^t\Z)$ and this forces us to change Axiom IV.
Let us go back to the conjugation sphere  $S^2$ examined in Example~\ref{Hopf}. The equivariant Leray-Hirsh Theorem, 
Corollary~\ref{cor lerayhirsh}, asserts that
 $\mathcal{H}^\ast_{C_2}(S^2;{}^t\mathbf{Z})$ is isomorphic as a $\Z[a]/(2a)$-module to $\Z[a]/(2a)\oplus \Z[a]/(2a)\langle \widetilde h \rangle$, 
 where $\widetilde h$ is of degree $2$ and restricts to a generator $h$ of $\HH^2(S^2;\mathbf{Z})$. The new form of Axiom IV is:

\begin{enumerate}
\item[IV'] If $\eta$ denotes the canonical bundle over $\mathbf{C}P^1 = S^2$, with the canonical Real 
structure, then $c(\eta) = 1 + \widetilde h$, where $1$ and $\widetilde h$ are generators of degree $0$ and $2$ of the $\Z[a]/(2a)$-module 
$\mathcal{H}^\ast_{C_2}(S^2;{}^t\mathbf{Z})$.
\end{enumerate}

Axioms I, II, III and IV' determine uniquely such equivariant Chern classes. The proof is analogous
to that for classical Chern classes, \cite[p. 58]{MR0202713} and is left to the interested reader. Notice that this requires the use of the splitting principle, which follows from \cite[Theorem 2.1]{MR0206940}. We thus proceed with the construction of these classes and show they satisfy all four axioms.

\begin{definition}\label{def inequivChernclasses}
{\rm Let $\eta$ be a Real bundle over the space $X$ with (equivariant) classifying
map $f : X \rightarrow BU$. Consider the classes $\widetilde c_n = \widetilde{\sigma}(c_n) \in \mathcal{H}^\ast_{C_2}(BU; {}^t\Z)$
and pull them back along $f^*$. The equivariant cohomology classes $\widetilde{c}_n(\eta) = 
f^\ast(\widetilde{c}_n) \in \mathcal{H}^\ast_{C_2}(X; {}^t\Z)$ are called the \emph{equivariant Chern classes} of $\eta$.}
\end{definition}

The following result is our version of Kahn's \cite[Th\'eor\`eme~2]{MR877194}.

\begin{theorem}
\label{thm:existence}
The equivariant Chern classes satisfy Axioms I, II, III, and IV'.
\end{theorem}

\begin{proof}
All axioms, except Axiom III, are merely routine. Axiom I follows from the fact that 
$\widetilde{\sigma}$ is a ring homomorphism. In particular the total equivariant
Chern class of the trivial bundle is~$1$.
Naturality (Axiom II) is a direct consequence of the existence of a universal bundle, and the
proof of Axiom IV' is essentially contained in the observation that lead us to modify Axiom IV.
We are thus left with Axiom III about the total equivariant Chern class of a Whitney sum. As this
class is defined by applying the section $\widetilde \sigma$ to the usual Chern class, we have basically
to check that the construction of the Whitney sum behaves well with respect to the $\HH^*$-frame
of a conjugation space.

Let $\xi$ and $\eta$ be two Real bundles over a compact space $X$ with classifying maps $f$ and $g$ 
respectively. Then a classifying map for the Whitney sum $\xi \boxplus \eta$ is the composite 
$\mu \circ (f \times g) \circ \Delta$, where $\mu$ denotes as above the H-structure map on 
$BU$ that gives rise to the Whitney sum, and $\Delta$ is the diagonal map for the space $X$.
We have thus to compute the image of $\widetilde c_n$ through
\[
\HH^{2n}_{C_2}(BU;\Z(n)) \xrightarrow{\mu} \HH^{2n}_{C_2}(BU \times BU; \Z(n))   \xrightarrow{(f\times g)^*} 
\HH^{2n}_{C_2}(X \times X; \Z(n)) \xrightarrow{\Delta^*}  \HH^{2n}_{C_2}(X; \Z(n))
\]
{}From Lemma \ref{lem factmu} we understand the first morphism on $\widetilde{c}_n$: It is the cross product
$\delta^*(\sum \widetilde c_i \times \widetilde c_{n-i})$. We are thus lead to compute
$(f \times g)^*(\widetilde c_i \times \widetilde c_{n-i}) = \widetilde c_i (\xi) \times \widetilde c_{n-i} (\eta)$
to which we must apply the composite
\begin{equation*}
\begin{split}
H^{2i}_{C_2}(X; \Z(i)) \otimes H^{2(n-i)}_{C_2}(X; \Z(n-i)) \xrightarrow{\times} H^{2n}_{C_2 \times C_2}(X \times X; \Z(i) \otimes \Z(n-i)) \\
\xrightarrow{\delta^*} H^{2n}_{C_2}(X \times X; \Z(n)) \xrightarrow{\Delta^*} H^{2n}_{C_2}(X; \Z(n))
\end{split}
\end{equation*}
But this is the product in the cohomology of $X_{hC_2}$ with twisted coefficients. Hence $\widetilde c(\xi \boxplus \eta)
= \widetilde c(\xi) \widetilde c(\eta)$.
\end{proof}

\section{Real bundles, complex bundles, and real bundles}
\label{sec relation}
In this short section we make the relation between equivariant Chern classes, classical Chern classes, and Stiefel-Whitney
classes explicit. A Real bundle can always be considered as a complex bundle by forgetting the involution. The universal 
equivariant Chern classes have been constructed by applying a section to the universal Chern classes. The following
proposition is thus obvious and recorded for completeness. 

\begin{proposition}
\label{propChern}
Let $\eta$ be a Real bundle over $X$. Then the image of 
$\widetilde{c}_n(\eta)$ via $\mathcal{H}^{2n}_{C_2}(X;{}^t\Z) \rightarrow \HH^{2n}(X;\mathbf{Z})$ is $c_n(\eta)$. \hfill{\qed}
\end{proposition}

We analyze now the relation between the equivariant Chern classes of a Real bundle over $X$ and the
Stiefel-Whitney classes of the associated real bundle over the fixed points $X^\tau$.
The following proposition generalizes \cite[Theorem~4]{MR877194}, 
which deals with spaces with trivial involution, as well as \cite[Proposition~6.8]{MR2171799} which treats the case of bundles over 
spherical conjugation spaces. It gives a description of the images of the equivariant Chern classes through the homomorphism
\[
\bar r: \mathcal{H}^\ast_{C_2}(X;{}^t\Z) \xrightarrow{r} \mathcal{H}^\ast_{C_2}(X^\tau; {}^t \Z) \rightarrow \HH^\ast_{C_2}(X^\tau; \F_2) \cong \HH^\ast (X^\tau; \F_2)[u].
\]

\begin{proposition}
\label{propChernSW2}
Let $\eta$ be a Real bundle over $X$ and $\eta^\tau$ the associated real bundle of fixed points. Then the image of 
$\widetilde{c}_n(\eta)$ in $H^\ast(X^\tau;\mathbf{F}_2)[u]$ is $\displaystyle \sum_{i=0}^n Sq^i(w_n(\eta^\tau))u^{n-i}$.
\end{proposition}

\begin{proof}
This follows at once from the analogous computation we have done for $BU$ in Proposition~\ref{propChernSW}.
If $f:X \rightarrow BU$ is an equivariant map classifying the Real bundle $\eta$, the equivariant Chern classes
are obtained by pulling-back the universal ones through $f$ and the Stiefel-Whitney classes of $\eta^\tau$
are obtained by pulling-back the universal ones through $f^\tau: X^\tau \rightarrow BO$.
\end{proof}

We recover, in our context, Kahn's main result result \cite[Th\'eor\`eme~4]{MR877194} for Real bundles over spaces with trivial involution.

\begin{corollary}
Let $\eta$ be a Real bundle over a space $X$ with trivial involution and $\eta^\tau$ the associated real bundle of fixed points. Then the mod
$2$ reduction of $\widetilde{c}_n(\eta)$ in $H^\ast(X;\mathbf{F}_2)$ is $w_n(\eta)$.  \hfill{\qed}
\end{corollary}

\begin{remark}
\label{ex:point}
{\rm We have seen that the equivariant Chern classes of a Real bundle determine the classical Chern classes by 
forgetting the $C_2$-action and the Stiefel-Whitney classes of the fixed point bundle by reducing mod $2$. In fact,
the pull-back diagram we have used in the proof of Theorem~\ref{thm intliftsigma} shows that these two sets of classes
determine the equivariant Chern classes, as long as one works with a Real bundle over a conjugation space.
In particular, the equivariant Chern classes of a Real bundle over a point are all zero.
}\end{remark}

\appendix
\section{Stiefel-Whitney classes and Thom spaces}
A particular case where all the above applies is that of a conjugation manifold, for then the tangent bundle is a Real bundle. 
For this case our results show a nice  interplay between the equivariant Chern classes of  the conjugation manifolds, its classical
Chern classes, as well as the the Stiefel-Whitney classes of the fixed submanifold (i.e. the tangent bundle on the fixed manifold). 
There is one set of classes missing from this picture, it is the Stiefel-Whitney classes of the conjugation manifold itself!
We thus end this paper with some remarks about them.  Let $M$ be a conjugation manifold of dimension $2n$, and 
$N= M^\tau$ the submanifold of fixed points. The isomorphism $\kappa$ in the $\HH^*$-frame relates the mod $2$ 
cohomology of $M$ in even degrees with that of $N$. Recall that the \emph{Wu classes} are defined as the unique
classes $v_k \in \HH^k(M; \F_2)$ such that for all $x \in \HH^{2n-k}(M; \F_2)$, 
$v_k \cup x  = Sq^k(x)$. We show that $\kappa$ behaves well with respect to both Wu and Stiefel-Whitney classes, a
result which has been proved as well by Hambleton and Hausmann in \cite[Proposition~2.9]{HHMathZ}.

\begin{theorem}\label{thm stiefel}
Let $M$ be a conjugation manifold of dimension $2n$.
Let $v_\ast^M $and  $w_\ast^M$ (resp. $v_\ast^N$ and  $w_\ast^N$) denote the Wu and  Stiefel-Whitney class of  $M$ (resp. of $N$) . 
Then, for any $k \geq 1$,
\[
\kappa(v_{2k}^M) =  v_k^N \quad \text{ and } \quad \kappa(w_{2k}^M) = w_k^N.
\]
\end{theorem}
\begin{proof} 
The isomorphism $\kappa$ preserves cup products and Steenrod squares so that:
\[
\begin{array}{rcl}
\kappa(v_{2k}^M) \cup x   & = &  \kappa (v_{2k}^M \cup \kappa^{-1}(x) )  \\
& = & \kappa (Sq^{2k}(\kappa^{-1}(x))) \\
& = & Sq^{k}(x)\\
\end{array}
\]
for all $x \in \HH^{n -k}(N; \F_2)$. The uniqueness of the Wu classes implies that $\kappa(v_{2k}^M) = v^N_k$ for all $k \in \mathbf{N}$. If $v$ denotes
the total Wu class,  $w$ the total Stiefel-Whitney class of the tangent bundle, and $Sq$ the total Steenrod square, 
then it is known that $Sq(v) =w$. Again, as $\kappa$ commutes with Steenrod squares, we see that
\[
w^N = Sq(v^N) = Sq(\kappa(v^M) )= \kappa(Sq(v^M)) = \kappa(w^M).
\]
\end{proof}

Another way to relate the Stiefel-Whitney classes of $M$ with those of $N$ would be to analyze their construction 
via Thom spaces. We will not go through all the details, but indicate why both the tangent bundle and the Thom space of a conjugation
manifold is still equipped with a conjugation structure. Of course 
the tangent bundle of $M$ is not anymore a closed space, but as it is homotopy equivalent to $M$, one can check directly  that the 
cohomological conditions are trivially satisfied as they match those satisfied by $M$.

\begin{theorem}\label{thm tnagentfixe}
Let $M$ be a conjugation manifold and $N$ the fixed submanifold, then $TN = TM^{\tau}$.
\end{theorem}

\begin{proof}
First choose a $C_2$-invariant Riemann metric on $M$. A fixed element in $TM$ is necessarily an element of $TM\vert_N = TN \oplus \nu_N$, 
where $\nu_N$ denotes the normal bundle of $N$, and  the direct sum is orthogonal and compatible with the action of $C_2$. 
We know from the equivariant tubular neighborhood theorem that there exists a neighborhood of the zero section of $\nu_N$, say $\mathcal{V}_N$ 
such that the exponential map
\[
exp : \mathcal{V}_N  \rightarrow M
\]
is injective, and maps onto an open tubular neighborhood of $N$ in $M$. Moreover, as the underlying metric is invariant, 
this map is equivariant with respect to the natural actions on $\mathcal{V}_N$ and  $M$. In particular, if there exists a vector in 
$TM\vert_N$ not in $TN$ that is fixed, then, from the above orthogonal decomposition we get a non-zero vector in $\nu_N$ that is fixed, 
and therefore, via the exponential map, a fixed point outside $N$, a contradiction.
\end{proof}

It follows that the natural compactification of $TM$, namely the Thom space $Th(M)$ is a conjugation \emph{space} (not a conjugation manifold). 
The next result is an extension to the tangent bundle of $M$ of a previous 
result of Hausmann, Holm, and Puppe concerning Real bundles over conjugation spaces (see \cite[p. 946]{MR2171799}).

\begin{corollary}\label{cor thom}
Let $M$ be a conjugation manifold and $N$ the fixed submanifold, then the Thom space $Th(M)$ is a conjugation space with fixed subspace $Th(N)$.
\end{corollary}
\begin{proof}
The proof is exactly the same as \cite[Proposition~6.4]{MR2171799}.
\end{proof}

For a manifold with involution, the Stiefel-Whitney numbers of the manifold are determined by the Stiefel-Whitney numbers of the 
fixed point submanifolds and of their normal bundles (see for instance \cite{MR516213}). In view of the above result for conjugation 
manifolds we have a slightly stronger result, compare also with \cite[Corollary~2.14]{HHMathZ}.

\begin{corollary}\label{cor cobord}
Let $M_1$ and $M_2$ be conjugation manifolds with fixed submanifolds $N_1$ and $N_2$ respectively. 
Then $M_1$ and $M_2$ are (non-equi\-va\-ri\-antly) cobordant if and only if $N_1$ and $N_2$ are cobordant.
\end{corollary}

\begin{remark}\label{rem BP}
{\rm Brown and Peterson determined in \cite{MR0146850} (see also the more complete versions \cite{MR0163326} and 
\cite{MR0176490}) all relations between Stiefel-Whitney classes which hold in a given degree for any $n$-dimensional
manifold. In our search for obstructions to realizability of conjugation manifolds with a given fixed point submanifold, see also \cite{PSpreprint},
we noticed that there are no obstructions to be found in terms of such relations. More precisely, if $R$ is a relation involving only
even degree Stiefel-Whitney classes of $2n$-dimensional manifolds, then the corresponding relation $r$ obtained by halfing
all degrees will also be true for any $n$-dimensional manifold.
}
\end{remark}


\bibliographystyle{plain}\label{biblography}

\begin{thebibliography}{10}

\bibitem{MR0206940}
M.~F. Atiyah.
\newblock {$K$}-theory and reality.
\newblock {\em Quart. J. Math. Oxford Ser. (2)}, 17:367--386, 1966.

\bibitem{MR652459}
E.~H. Brown, Jr.
\newblock The cohomology of {$B{\rm SO}_{n}$} and {$B{\rm O}_{n}$} with integer
  coefficients.
\newblock {\em Proc. Amer. Math. Soc.}, 85(2):283--288, 1982.

\bibitem{MR0146850}
E.~H. Brown, Jr. and F.~P. Peterson.
\newblock Relations between {S}tiefel-{W}hitney classes of manifolds.
\newblock {\em Bull. Amer. Math. Soc.}, 69:228--230, 1963.

\bibitem{MR0163326}
E.~H. Brown, Jr. and F.~P. Peterson.
\newblock Relations among characteristic classes. {I}.
\newblock {\em Topology}, 3(suppl. 1):39--52, 1964.

\bibitem{MR0176490}
E.~H. Brown, Jr. and F.~P. Peterson.
\newblock Relations among characteristics classes. {II}.
\newblock {\em Ann. of Math. (2)}, 81:356--363, 1965.

\bibitem{MR1324339}
K.~S. Brown.
\newblock {\em Cohomology of groups}, volume~87 of {\em Graduate Texts in
  Mathematics}.
\newblock Springer-Verlag, New York, 1994.
\newblock Corrected reprint of the 1982 original.

\bibitem{MR2084395}
P.~F. dos Santos and P.~Lima-Filho.
\newblock Quaternionic algebraic cycles and reality.
\newblock {\em Trans. Amer. Math. Soc.}, 356(12):4701--4736, 2004.

\bibitem{MR2198191}
M.~Franz and V.~Puppe.
\newblock Steenrod squares on conjugation spaces.
\newblock {\em C. R. Math. Acad. Sci. Paris}, 342(3):187--190, 2006.

\bibitem{HHMathZ}
I.~Hambleton and J.-C. Hausmann.
\newblock Conjugation spaces and 4-manifolds.
\newblock {\em Math. Zeitschrift}, 269:521--541, 2011.

\bibitem{MR2171799}
J.-C. Hausmann, T.~Holm, and V.~Puppe.
\newblock Conjugation spaces.
\newblock {\em Algebr. Geom. Topol.}, 5:923--964 (electronic), 2005.

\bibitem{MR0202713}
F.~Hirzebruch.
\newblock {\em Topological methods in algebraic geometry}.
\newblock Third enlarged edition. New appendix and translation from the second
  German edition by R. L. E. Schwarzenberger, with an additional section by A.
  Borel. Die Grundlehren der Mathematischen Wissenschaften, Band 131.
  Springer-Verlag New York, Inc., New York, 1966.

\bibitem{MR877194}
B.~Kahn.
\newblock Construction de classes de {C}hern \'equivariantes pour un fibr\'e
  vectoriel r\'eel.
\newblock {\em Comm. Algebra}, 15(4):695--711, 1987.

\bibitem{MR516213}
C.~Kosniowski and R.~E. Stong.
\newblock Involutions and characteristic numbers.
\newblock {\em Topology}, 17(4):309--330, 1978.

\bibitem{MR639802}
M.~Nagata, Goro Nishida, and H.~Toda.
\newblock Segal-{B}ecker theorem for {$KR$}-theory.
\newblock {\em J. Math. Soc. Japan}, 34(1):15--33, 1982.

\bibitem{PSpreprint}
W.~Pitsch and J.~Scherer.
\newblock Realization of conjugation spaces, in preparation.

\end{thebibliography}

\end{document}